\documentclass[12pt, reqno]{amsart}
\usepackage{amscd} 
\usepackage{amsfonts} 
\usepackage{amssymb} 
\usepackage{latexsym} 
\usepackage[arrow, matrix, curve]{xy}
\usepackage{color}

\newcommand{\ncm}{\newcommand}

\newtheorem{theorem}{Theorem}[section]
\newtheorem{prop}[theorem]{Proposition}

\newtheorem{lemma}[theorem]{Lemma}
\newtheorem{cor}[theorem]{Corollary}
\newtheorem{lem&def}[theorem]{Lemma \& Definition}

\newtheorem{remark}[theorem]{Remark}

\def\Z{\mathbb{Z}\,} 
 
\def\N{\mathbb{N}\,}

\ncm{\Ann}{\mbox{\rm Ann}\,}
\ncm{\End}{\mbox{\rm End}\,}

\def\Indec{\mbox{\rm Indec}\,}

\def\|{\, | \,}

\def\bra{\langle}
\def\ket{\rangle}

\ncm{\rarr}[1]{\stackrel{#1}{\longrightarrow}}
\ncm{\larr}[1]{\stackrel{#1}{\longleftarrow}}
\def\cop{\Delta}

\def\eps{\varepsilon}

\def\-2{_{(-2)}}
\def\-1{_{(-1)}}
\def\0{_{(0)}}
\def\1{_{(1)}}
\def\2{_{(2)}}
\def\3{_{(3)}}

\def\du1{\hat 1}

\newcommand{\ZZ}{\mathbb{Z}}  
\newcommand{\CC}{\mathbb{C}}

\newcommand{\A}{\mathsf{A}}

\def\cent1{\mathsf{Z}(\A_1)}

\newcommand{\ee}[1]{\mathsf{e}_{#1}}
\newcommand{\uu}[1]{\mathsf{t}_{#1}} 
\newcommand{\ww}[1]{\mathsf{w}_{#1}} 

\begin{document}
\title[Depth of half quantum group in Drinfeld double]{A quantum subgroup depth}
\author[A.~Hernandez, L.~Kadison and S.A.~Lopes]{Alberto Hernandez, Lars Kadison and Samuel A.\ Lopes} 
\address{Departamento de Matem\'atica \\ Faculdade de Ci\^encias da Universidade do Porto \\ 
Rua do Campo Alegre, 687 \\ 4169-007 Porto, Portugal} 
\email{ahernandeza079@gmail.com, lkadison@fc.up.pt, slopes@fc.up.pt} 
\date{} 

\begin{abstract}
 The Green ring of  the half quantum group $H=U_n(q)$ is computed in \cite{COZ}.   The tensor product formulas between indecomposables may be used for a generalized subgroup depth computation in the setting of quantum groups --  to  compute depth of the Hopf subalgebra $H$ in its Drinfeld double $D(H)$. In this paper the Hopf subalgebra quotient module $Q$ (a generalization of the permutation module of cosets for a group extension) is computed and, as $H$-modules, $Q$ and its second tensor power  are decomposed into a direct sum of indecomposables.  We note that the least power $n$, referred to as depth, for which $Q^{\otimes (n)}$ has the same indecomposable constituents as $Q^{\otimes (n+1)}$  is $n = 2$,  since $ Q^{\otimes (2)}$ contains all $H$-module indecomposables, which determines the minimum even depth
$d_{ev}(H,D(H)) = 6$.  
\end{abstract} 
\maketitle

\section{Introduction}

The notion of depth has historical roots in the theory of operator algebras, with its most natural definition being in the  ring-theoretic terms of subrings, balanced tensors and bimodules (see the end of Section~2). Specializing to group algebra extensions, the minimum depth of a subgroup $K \leq G$ of a finite group $G$, over a field of characteristic zero, for example, is two if $K$ is normal in $G$, 
or more strongly one if $G=KC_G(K)$, and the odd number $2n+1$ if all simples occur in $n$ applications of  induction-restriction applied to the unit simple (or trivial module), $4$ and $6$ in certain exceptional cases, and it is not known if minimum depth $2n > 6$ can occur.  Also, the (minimum) depth is less than or equal to twice the number of conjugates of $K$ that intersect to equal the core$_G(K)$ \cite[Theorem 6.9]{BKK}. 
The subgroup depth, over $\CC$, of the series of permutation groups $S_n \subseteq S_{n+1}$ is computed
to be $2n - 1$ in \cite{BKK}, having the same value over any  commutative base ring \cite{BDK}.  With a twist automorphism introduced, the depth of the corresponding twisted group algebra series shrinks to $ 2\left(n - \left\lceil \frac{\sqrt{8n+1} - 1}{2} \right\rceil \right) + 1$, as computed in \cite{D}:  a general reason for why depth should shrink when twisting is introduced is given in \cite{HKS}.  The subgroup depth of the alternating group series $A_n \subset A_{n+1}$ is $2\left(n - \lceil \sqrt{n}\,  \rceil \right)+1$ over the complex numbers \cite{BKK} (a precise value is not known in positive characteristic, although it is between the value just given and  $2n - 3$ \cite{BDK}).  
Other results on subgroup depth can be seen in \cite{F1, F2, FKR, HHP, HHP2}. 

The paper \cite{K2014} shows how subgroup depth is determined by the tensor powers of the permutation module of cosets:  this extends to determining subalgebra depth of a Hopf subalgebra $R \subseteq H$ from  its quotient module $Q$ as an $R$-module (see Section~\ref{S:two} for the definition of $Q$).  The depth of $Q_R$  is determined from the least $n$ for which $Q \otimes \cdots \otimes Q$ ($n$ times $Q$, denoted by $Q^{\otimes (n)}$) has the exact same indecomposable constituents as $Q^{\otimes (n+1)}$ (a  chain of subsets  increasing with  $n$).  

Drinfeld's quantum double $D(H)$  of a Hopf algebra $H$ is frequently applied to a group algebra.  The depth of a finite group $G$ in its double $D(G)$ (over $\CC$) is studied in \cite{HKY}, where it is shown to be closely related to the tensor power of the adjoint representation of $G$ on itself at which it is faithful, a topic introduced and explored in \cite{P}. For example, in \cite{P} it is proved that the adjoint representation of the symmetric group on itself is faithful, which rephrased in the terms of $Q$ in this paper, shows that $Q$ and $Q^{\otimes (2)}$ have the same indecomposable constituents, or $Q_G$ has ``depth $1$'' for the group $\CC$-algebra of a symmetric group on $3$ or more letters (in its Drinfeld double with quotient module $Q$).
It is also noted in \cite[Lemma 1.3]{P} that for $G$ set equal to certain semidirect products of elementary abelian $p$- and $q$-groups, where $p,q$ are primes such that $p \| q-1$, the adjoint action of $G$ on itself is not faithful: from this recipe, a group $G$ of order $108$ with $Q$ of depth $2$ in $D(G)$ is constructed in \cite[Example 6.5]{HKY}.  For any semisimple Hopf subalgebra pair $R \subseteq H$, the depth of $Q$ coincides with the length $n$ of the  chain of annihilator ideals of tensor powers of $Q$, i.e., the least $n$ for which $\Ann Q^{\otimes (n)}$ is a Hopf ideal 
 \cite[Theorem 3.14]{K}, \cite{H}.  

The half quantum groups, or Taft algebras $U_n(q)$, are $n^2$-dimensional algebras generated by a group-like element of order $n$ and a nilpotent element of order $n$, with an anti-commutation relation between these two generators involving a primitive $n$-th root-of-unity $q$ in the base field. Like  group algebras, they are Hopf algebras, but unlike complex finite group algebras, they are not semisimple nor cocommutative.  The noncocommutativity is explicit in the skew primitivity of the nilpotent element.  The Drinfeld doubles of the half quantum groups are reduced to the transparent terms of generators and relations in \cite{C}.   The Green rings (or representation rings) of the half quantum groups are determined in \cite{COZ} by means of a computation in \cite{G} of the tensor products of the $n^2$ isoclasses of indecomposables:  the Green ring is shown to be commutative, although the half quantum groups are  not quasitriangular (nor almost cocommutative \cite{CC}).   

Two finite-dimensional modules over a finite-dimensional algebra are similar, denoted with a $\sim$, if they have the same nonzero indecomposable constituents.
In this paper, the quotient module $Q$ for the Taft algebra $U_n(q)$, for any $n \in \N$, $n \geq 2$ and $q$ a primitive $n$-th root of unity in the base field,  is computed in Section~\ref{S:four}, decomposed in Theorem~\ref{T:KSdecQ}, and  its second tensor power is shown in  Theorem~\ref{T:decQQ} to contain all indecomposables of $U_n(q)$, a very strong form of faithfulness; the conclusion is that  $Q \, {\not \sim} \,   Q \otimes Q$, but $Q^{\otimes (2)} \sim
Q^{\otimes (3)}$ as $U_n(q)$-modules. Thus, the depth of $Q$ is $2$, which translates, using \cite[Theorem 5.1]{K2014}, into a minimum even depth $d_{ev}(U_n(q), D(U_n(q))) = 6$ for this Hopf subalgebra pair. 

\section{The general quotient module $Q$ in more detail}\label{S:two}
We introduce the quotient module of certain Hopf algebra-subalgebra pairs, note a generalization of a relative Maschke's theorem \cite[Theorem 3.7]{K} and point out some differences between the  quotient module and its restriction to the subalgebra.  The material in this section is mostly of theoretical interest and some of it may be skipped in reading the main results in Sections~4 and~5.  

Throughout the paper, let $H$ be a finite-dimensional Hopf algebra over a field $k$. Suppose $R$  is a left coideal subalgebra of $H$, i.e., $R$ is a subalgebra and the coproduct satisfies $\cop(R) \subseteq H \otimes R$.  The subalgebra $R$ is augmented by the counit $\eps$ of $H$; let $R^+$ denote the augmentation ideal.  
Define the quotient $H$-module \begin{equation}
Q = H/ R^+H.
\end{equation}
E.g., note that $\overline{r} = \overline{1_H} \eps(r)$
for all $r \in R$, with usual coset notation.  In later sections, we will focus on the $R$-module $Q$ resulting from restriction, but we point out some differences between  $Q_R$ and the cyclic module $Q_H$  in this section. After the next theorem our focus for $R$ falls back to the more restrictive notion of Hopf subalgebra, where $\cop(R) \subseteq R \otimes R$ and $S(R) = R$: thus, $R$ is a Hopf algebra with a restriction of the structure of $H$.  

The quotient module $Q$ is 
also a coalgebra, since $R^+H$ is a coideal, by elementary considerations. 
It satisfies the identity of a right $H$-module coalgebra: $\cop(qh) = q\1 h\1 \otimes q\2 h\2$
 and $\eps_Q(qh) = \eps_Q(q) \eps(h)$, for every $q \in Q, h \in H$. The canonical epimorphism of right $H$-module coalgebras $H \rightarrow Q$ is denoted by $h \mapsto \overline{h}$. 

Let $A$ be any ring, and $B$ be a unital subring of $A$.  Recall that $A$ is a right semisimple extension of  $B$ if every right $A$-module $M$ splits in the kernel exact sequence of the canonical epimorphism $M \otimes_B A 
\rightarrow M$; left semisimple extensions are similarly defined with left modules. Equivalently, short exact sequences of $A$-modules that are $B$-split, also split as $A$-modules.   Note that the $A$-module $M$ is isomorphic to a direct summand of $M \otimes_B A$,
a fact which we denote by $M \| M \otimes_B A$, and $M$ is therefore $B$-relative projective.  A third equivalent condition for semisimple extensions is that all modules are relative projective.  (Note that a projective module $P_A$ has the property that $P \otimes_B A \rightarrow P$ splits for any subalgebra $B$ in $A$.) 

A ring $A$ is a separable extension of a subring $B$ if for every $A$-module $M$ there is a splitting of the canonical epimorphism $M \otimes_B A  \rightarrow M$, natural with respect to $A$-module mappings
$M \rightarrow N$. Equivalently, there is an element
$e = e^1 \otimes_B e^2 \in A \otimes_B A$ satisfying $e^1 e^2 = 1_A$ and $ae = ea$ for all $a \in A$ (suppressing a possible summation $\sum_{(e)} e^1 \otimes_B e^2$, using Sweedler-type notation).  

 We note the following theorem, which generalizes Maschke's theorem for group and Hopf algebras, and has the same proof as in \cite{K}, so the details are omitted. 

\begin{theorem}\label{T:21}
A finite-dimensional Hopf algebra $H$ is a right semisimple extension of a left coideal subalgebra $R$ iff $k_H$ is isomorphic to a direct summand of $Q_H$ iff $k_H$ is $R$-relative projective
iff there is a $q \in Q$ such that $\eps_Q(q) = 1$ and $qh = \eps(h) q$ for each $h \in H$ iff there is $s \in H$ such that $\eps(s) = 1$ and $sH^+ \subseteq R^+H$ iff
$H$ is a separable extension of $R$.
\end{theorem}
\begin{proof}
The module structure on $k_H$ is of course $1 . h = \eps(h)$.  Note that $Q \cong k \otimes_R H$ via $1 \otimes h \mapsto \overline{h}$; thus $Q$ is $R$-relative projective by a standard exercise.  Moreover, for any right $H$-module $M$, the $H$-module $M \otimes Q$ is given by the diagonal action $(m \otimes q)h = mh\1 \otimes qh\2$, and we have the natural isomorphism of right $H$-modules,
\begin{equation}
\label{eq: induct}
M \otimes_R H \cong M \otimes Q
\end{equation}
given by  $m \otimes_R h \mapsto mh\1 \otimes \overline{h\2}$ with inverse
$m \otimes \overline{h} \mapsto mS(h\1) \otimes_R h\2$. 
The counit of $Q$ is $R$-split by $\lambda \mapsto \lambda \overline{1_H}$ as noted above.   If $k_H$ is relative projective, the counit of $Q$ splits over $H$.
The element $q = \overline{s}$ is the image of a splitting $\sigma: k_H \rightarrow Q_H$.  
It follows from the natural isomorphism \eqref{eq: induct} that every module is $R$-relative projective
and indeed that $S(s\1) \otimes_R s\2$ is a separability element.
\end{proof}
Note that Eq.~(\ref{eq: induct}) implies that all tensor powers of $Q$ are relative projective, and that the class of relative projective $H$-modules enjoys the ideal property that $M \otimes V$ is relative projective if $V$ is relative projective, and $M$ is any module.  The property of separability, or its absence, is usually easy to detect in a subalgebra pair; e.g., global dimension satisfies  $\operatorname{gldim}(B) \geq \operatorname{gldim}(A)$ for a projective separable extension $B \subseteq A$.  Thus a nonsemisimple Hopf algebra is never a semisimple extension of a semisimple Hopf subalgebra.  Furthermore, a separable extension of Hopf algebras is an ordinary Frobenius extension, where the modular functions are related by restrctiction \cite[Corollary 3.8]{K}. Thus, a non-unimodular Hopf algebra (like the Taft algebra) does not form a semisimple extension with its Drinfeld double (by \cite[Radford's Theorem 6.10]{NEFE}). 

Below we use the notation $V_H \| W_H$ (suggested by the Krull-Schmidt Theorem) if $V$ is a direct summand of $W$ up to module isomorphism, sometimes also denoted by
$V_H \oplus * \cong W_H$. 
\begin{cor}
Suppose a finite-dimensional Hopf algebra extension $H \supseteq R$ is not a semisimple extension. Then $k_H$ is not a direct summand of any tensor power of $Q_H$ up to isomorphism
(i.e.,  $k_H {\not  |} Q^{\otimes (n)}$ for any $n \in \N$).   If $W_H$ satisfies $k_H \| W^{\otimes (n)}$
for some $n \in \N$, then $W_H {\not |} Q^{\otimes (m)}$ for any $m \in \N$.  
\end{cor}
\begin{proof}
A direct summand of a relative projective, such as $Q_H$, is relative projective.  By Theorem~\ref{T:21}, the extension is semisimple if $k_H$ is relative projective.
\end{proof}
   The restricted module $Q_R$ then is very different than $Q_H$, since $k_R \| Q_R$ always, as noted above.  We will focus only on the restricted module $Q_R$ below. 

\subsection{The quotient module and quantum subgroup depth}  
For any ring $A$, and $A$-module $X$, let $1 \cdot X = X$, $2 \cdot X = X \oplus X$, etc.  The similarity relation $\sim$ mentioned in the introducion  may be defined on $A$-modules as follows.  Two $A$-modules $M,N$ are similar, written $X \sim Y$, if $X \| n \cdot Y$ and $Y \| m \cdot X$ for some positive integers $m, n$.  This is an equivalence relation, and carries over to isoclasses in the Grothendieck group of $A$, or the Green ring if $A$ is a (quasi-) Hopf algebra.  If $M \sim N$ and $X$ is an $A$-module,
then we have $M \oplus X \sim N \oplus X$; if $\otimes$ is a tensor on mod-$A$, then also $M \otimes X \sim N \otimes X$. In case $A$ is a finite-dimensional algebra, $M \sim N$ if and only if $\Indec (M) = \Indec (N)$, where $\Indec (X)$ denotes the set of isoclasses of the indecomposable module constituents of $X$ in its Krull-Schmidt  decomposition.  

If $A $ is a finite-dimensional Hopf algebra, and $r(A)$ denotes the Green ring of $A$, then $r(A)$ has a preferred basis as a $\Z$-algebra given by the isoclasses of all indecomposable $A$-modules; the algebra is of finite rank iff
$A$ has finite representation type.  The set of nonzero linear combinations of indecomposables with non-negative integer coefficients we call the \textit{positive quadrant}; these elements correspond to the (isoclasses of) actual $A$-modules.  The
relations $\|$ and $\sim$ are meaningful in the positive quadrant: say $a,b,c$ are elements there, then $a \sim b$
(so $a \| nb$ and $b \| ma$ for some $m, n \in \N$) 
implies $a + c \sim b + c$ (but not conversely unless $\Indec (c) \cap \Indec (a) = \emptyset = \Indec (c) \cap \Indec (b)$),  $ac \sim bc$, $3a \sim 5 b$, and others.  Also, $a \sim b$ and
$c \sim d$ implies $ac \sim bd$ and $a + c \sim b + d$.  

Given an $a \in r(A)$ in the positive quadrant, define for each $m \geq 1$, the polynomial $p_m(a) = a + \cdots + a^m$, noting that
$p_m(a) \| p_{m+1}(a)$. Let $p_0(a) = 1$.  
The \textit{depth} of $a$ (or any module in its isoclass) is the least $n \in \N$ for which
$p_n(a) \sim p_{n+1}(a)$ if such exists, denoted by $d(a) = n$; otherwise $d(a) = \infty$. Note that $p_n(a) \sim p_{n+1}(a)$ implies that $p_n(a) \sim p_{n+r}(a)$, for any $r \in \N$, by an exercise.   Elements in $r(A)$ that represent algebras or coalgebras in the tensor category mod-$A$ of finite-dimensional $A$-modules, such as the quotient module $Q$ defined above for $A = H$ with Hopf subalgebra $R$, with isoclass denoted by  $\tilde{q}$, satisfy $\tilde{q}^i\| \tilde{q}^{i+1}$ for all $i \in \N$, a fact 
which follows from applying the counit and comultiplication, or the unit and multiplication. In this case,
$p_n(\tilde{q}) \sim \tilde{q}^n$, and the depth $n$ condition is replaceable by $\tilde{q}^n \sim \tilde{q}^{n+1}$.   If $A$ has finite representation type, it is  clear that an algebra or coalgebra such as $\tilde{q}$ has finite depth; and indeed any element $a$ in the positive quadrant has finite depth. 

Note that depth $d(\tilde{q}) = 0$ in the Green ring of $R$ if and only if
$R$ is a normal Hopf subalgebra of $H$, i.e. $R^+H = HR^+$. Then $\tilde{q} \sim 1$ by recalling the right module $Q = H/R^+H$.    

Suppose $r(A)$ is commutative, which is the case if $H$ is quasi-triangular or braided \cite{CK}. Let $a,b \in r(A)$ be two elements of finite depth in the positive quadrant.  Then $ab$ and $a + b$ have finite depth.   It is noted in \cite{F} that elements of finite depth are algebraic elements in the Green ring
of a cocommutative Hopf algebra,
and conversely by a paraphrasing;  i.e., $d(a) < \infty$ $\Leftrightarrow$ $g(a) = 0$ for some nonzero polynomial $g(x) \in \Z[x]$. 
Then $d(ab) < \infty$ and $d(a + b) < \infty$ follows in an exercise, or as standard textbook material on  the set of algebraic elements in a commutative ring forming a subring.  

Consider the special case of a subgroup pair $K \leq G$ where $G$ is a finite group. Let $R =kK$ and $H = kG$ be the group algebra extensions over a ground field $k$ of any characteristic; then $Q \cong k[K \setminus G]$, the right coset space, by  noting that $(1 - k')g  \in R^+H$ are the generators for each $k' \in K, g \in G$.  It is noted in \cite[Chapter 9]{F} that $\tilde{q} \in r(H)$  has finite depth, i.e., the isoclass of any permutation module, such as $\tilde{q}$, is an algebraic element in the Green ring of $G$.   

Based on simplifications made in the literature (see the preliminaries in \cite{K}), the \textit{subgroup depth} $d_k(K,G)$ is a positive integer equal to, or one less than, $d_{ev}(kK,kG) = 2d(Q_R) + 2$, the minimum even depth, and satisfying $$ d_h(kK,kG) - 2 \leq d_k(K,G) \leq d_h(kK,kG) + 1,$$
where the h-depth $d_h(kK,kG) = 2d(Q_H) + 1$.  The two equalities of depth are explained and proven in \cite[Theorem 5.1]{K2014} for any Hopf subalgebra $R \subseteq H$:
\begin{eqnarray}
\label{eq: h}
d_h(R,H) & = & 2d(Q_H) + 1; \\
\label{eq: ev}
d_{ev}(R,H) & = & 2d(Q_R) + 2.
\end{eqnarray}
The minimum odd depth is not determined by this approach; in case $H$ is semisimple, there is a nice graphical technique described in \cite{BKK} for its determination.  
The minimum h-depth, even depth and odd depth $d_{odd}(R,H)$ may be viewed as closely related (to within two) and on an equal footing:  subalgebra depth $d(R,H)$ may be even and equal to $d_{ev}(R,H)$, 
or odd and equal to $d_{odd}(R,H) = d_{ev}(R,H) - 1$.
 From Feit's theorem in group representation theory, it follows that  $d_k(K,G) < \infty$, a result with interesting upper bounds given in \cite{BDK}.

This paper will not make use of the most general and full definition of depth in \cite{BDK} and h-depth in \cite{K2012}.  For the reader's convenience, we sketch the general definition of depth, the minimum depth $d(B,A)$,and h-depth $d_h(B,A)$ of a subring $B$ in a ring $A$: they are defined  as follows by a positive integer, or infinity, in terms of the natural
$A$-$A$-bimodule structure on the tensor powers $A^{\otimes_B (n)}$ (where $n \in \N$ and  $A^{\otimes_B (0)} = B$), and its restrictions to the three other combinations of $B$- and $A$-bimodules. If
 \begin{equation}
\label{eq: finitedepth}
A^{\otimes_B (m)} \sim A^{\otimes_B (m+1)}
\end{equation}
 for a nonnegative  integer  $m$ for which this similarity holds in one of the four bimodule structures mentioned above, we say the subring $B \subseteq A$, or the ring extension $A \supseteq B$, has h-depth $2m-1= 1,3,5,\ldots$ in
the $A$-$A$-bimodule structure, the ring extension has depth  $2m =2,4,6,\ldots$ in the $A$-$B$- or $B$-$A$-bimodule structures, and the ring extension has depth 
$2m+1=1,3,5,\ldots$ in the $B$-$B$-bimodule structure.  Having depth $2m$ implies that the subring also has h-depth $2m+1$ by tensoring the similarity, and depth $2m+1$ by restriction of the similarity. Also, having h-depth $2m+1$ implies that the subring has depth $2m+2$, and having depth $2m+1$ implies having depth $2m+2$.  Then define  $d(B,A)$ as the minimum such natural number as well as the minimum h-depth $d_h(B,A)$, an odd positive integer, unless  there is no such $m \in \N$ satisfying Eq.~(\ref{eq: finitedepth}), in which case we may write
$d(B,A) = \infty = d_h(B,A)$. Note then that \begin{equation}
\label{eq: ineq}
-2 \leq d_h(B,A) - d(B,A) \leq 1
\end{equation}
 if one of these is finite (also noted above for group algebras).

For example, we note a shorter proof of \cite[Corollary 3.3]{K2014}, which states
that an h-separable Hopf algebra extension $H \supseteq R$ is a trivial extension.
For h-separability is equivalent to $d_h(R,H) = 1$, which is equivalent to the quotient module $Q_H$ having depth $0$, i.e., $Q \sim k_H$. Thus $\overline{h} = \overline{1_H}h = \eps(h)\overline{1_H}$ for all $h \in H$. But $Q$ is cyclic:
$Q = \{ \overline{h} \| h \in H \} = k\overline{1}$.  Thus, $\dim Q = 1 = \dim H/ \dim R$.  Hence, $R = H$.  
 
The next lemma follows from the considerations about general depth above and the inequality~(\ref{eq: ineq}), Eqs.~(\ref{eq: h}) and~(\ref{eq: ev}), and is left as an exercise. 
\begin{lemma}. 
The depth of the $H$-module $Q$ and its restricted module satisfy
$$ d(Q_H) - d(Q_R) \in \{ 0,1 \}. $$
Both values are attained. 
\end{lemma}
\begin{remark}
\begin{rm}
Restricting $Q^{\otimes (n)} \sim Q^{\otimes (n+1)}$ from $H$-modules to $R$-modules shows
that $d(Q_H) \geq d(Q_R)$.  The value $d(Q_H) - d(Q_R) = 1$ is very usual when $H \neq R$; however, $d(Q_H) = d(Q_R)$ occurs when $R,H$ are the complex group algebras of the alternating groups, $A_4, A_5$
\cite{BKK, BDK, K2012}. 
\end{rm}
\end{remark}
\section{Preliminaries on the half quantum group} 

We keep our previous assumption that $k$ is any field, but assume further that $0\neq q\in k$ is a primitive $n$-th root of unity with $n \geq 2$. The latter implies, in particular, that the characteristic of $k$ does not divide $n$. Define the Taft algebra by
$$U_n(q) = k\bra G,X \, \| G^n = 1, X^n = 0, GX = qXG \ket.$$
This is an $n^2$-dimensional algebra with basis $\{ G^i X^j \| i,j=0,1,\ldots,n-1 \}$.  It has a coalgebra structure given by $\cop(G) = G \otimes G$, $\cop(X) = X \otimes G + 1 \otimes X$, $\eps(G) = 1$, and $\eps(X) = 0$.  There is also a Hopf algebra antipode given by $S(G) = G^{-1}$ and $S(X) = -XG^{-1}$ \cite{SY}. 

The quantum double $D(U_n(q))$ is given in terms of generators and relations in \cite{C} by
\begin{eqnarray*}
D(U_n(q)) &=&  k\bra a,b,c,d \, \|  a^n = 0 = d^n,  b^n=1 = c^n,  ba = qab,   \\ 
& &  \hspace{-15pt}  dc = qcd, db = qbd, bc = cb, ca = qac, da - qad = 1 - bc \ket.
\end{eqnarray*}
A basis for $D(U_n(q))$ is given by $\{ a^i b^j c^r d^s \| i,j,r,s = 0,1,\ldots, n-1 \}$. 
The Hopf algebra structure in \cite{C} is not needed below, but shows that the  Hopf subalgebra
 generated by $a,b$ is isomorphic to $U_n(q)$.  We view this as the embedding $G \mapsto b$
and $X \mapsto a$ for the purposes of computing depth (a Morita invariant in terms of Morita invariance of ring extensions \cite{K}).  The Hopf algebra $D(U_n(q))$, and its quotient in  the small quantum groups, is further discussed in \cite[Chapter 9]{CK}, and for $n=2$ below in Example~\ref{E:51}.

\subsection{Indecomposable modules of $U_n(q)$}
The principal modules, or projective indecomposables, of the half quantum group $U_n(q)  = \bra a,b \, \| b^n = 1, a^n = 0, ba = qab \ket$ are determined from a basic set of primitive idempotents given by $e_i = \frac{1}{n}\sum_{j=0}^{n-1} q^{(i-1)j}b^j$, for $i =  1, \ldots, n$.
The projective indecomposables are thus the $n$-dimensional modules $P_1 = e_1U_n(q)$, $P_2 = e_2U_n(q), \ldots, P_n = e_n U_n(q)$.
Each $P_i$ is the projective cover of the one-dimensional simple module $S_i = P_i / P_i \operatorname{rad}(U_n(q))$, which has eigenvalue
$q^{n - i +1}$ from the action of the grouplike $b$ and is annihilated by $a$.  

The radical ideal $J=\operatorname{rad}(U_n(q))$ is generated by the nilpotent element $a$: $J = aU_n(q)$ with
$J^n = 0$, but $J^{n-1} \neq 0$.   
The Loewy length of each $P_i$ is $n$ and equal to its composition length: the algebra is Nakayama (or serial) \cite{SY}.  The composition series of each $P_i$ is given by $$P_i \supset P_i J \supset P_i J^2 \supset \cdots \supset P_i J^{n-1} \supset \{ 0 \} .$$
The indecomposable module isoclasses of $U_n(q)$ are represented by
$P_iJ^{r-1}$, for $i,r = 1,\ldots,n$ \cite{SY}.   

\subsection{Preliminaries on the $q$-binomial coefficients}

We recall the definition and some basic properties of the $q$-binomial coefficients, also known as Gauss polynomials. For any integer $j\geq 1$, set $(j)_q=1+q+\cdots +q^{j-1}$, $(j)!_q=(j)_q\cdots (1)_q$ and $(0)!_q=1$. Note that, for $q\neq 1$,  $(j)_q=0$ if and only if $q^j=1$.  Finally, define the $q$-binomial coefficients ${k\choose j}_q$ inductively as follows, for $k\geq j\geq 0$:
\begin{gather*}
{k\choose 0}_q=1={k\choose k}_q \quad \mbox{for $k\geq 0$,}\\[5pt]
{k\choose j}_q=q^j {k-1\choose j}_q+{k-1\choose j-1}_q \quad \mbox{for $k> j> 0$.}
\end{gather*}
These $q$-binomial coefficients are thus polynomials in $q$ with integer coefficients which agree with the corresponding binomial coefficients when $q=1$. Furthermore, whenever $k> j> 0$ and $(k-1)!_q\neq 0$ we have $\displaystyle {k\choose j}_q=\frac{(k)!_q}{(j)!_q (k-j)!_q}$.

It will be convenient to add the following notational conventions: for all $k\geq 0$ define $\displaystyle {k-1\choose -1}_q=0$, $\displaystyle {k\choose k+1}_q=0$ and $\displaystyle {-1\choose 0}_q=1$. With these conventions, the recurrence relation above can be extended to
\begin{equation}\label{E:recurrence}
{k\choose j}_q=q^j {k-1\choose j}_q+{k-1\choose j-1}_q \quad \mbox{for all $k\geq j\geq 0$.}
\end{equation}
We need a further fact concerning the $q$-binomial coefficients:

\begin{lemma}\label{L:qbinom}
Let $n\geq 2$  and assume $q$ is a primitive $n$-th root of unity. If $\alpha$ and $\beta$ are integers satisfying $\alpha\geq n$, $1\leq \beta\leq n-1$ and $\alpha-\beta<n$, then $\displaystyle {\alpha \choose \beta}_q=0$.
\end{lemma}

\begin{proof}
The proof is a straightforward induction on $\alpha\geq n$. Since $(n-1)!_q\neq 0$ we have  $\displaystyle {n\choose \beta}_q=\frac{(n)!_q}{(\beta)!_q (n-\beta)!_q}=0$. So the case $\alpha=n$ holds and we can assume $\alpha\geq n+1$. Thus  $\beta\geq 2$  is implied by the condition $\alpha-\beta<n$. The induction hypothesis thus yields $\displaystyle {\alpha-1\choose \beta}_q=0={\alpha-1\choose \beta-1}_q$, and the recurrence relation 
$\displaystyle{\alpha\choose \beta}_q=q^\beta {\alpha-1\choose \beta}_q+{\alpha-1\choose \beta-1}_q$ proves the inductive step.
\end{proof}

\section{The $U_n(q)$-module $Q$}\label{S:four}

We view the Taft algebra $U_n(q)$ as the Hopf subalgebra of its quantum double 
\begin{align*}
D(U_n(q))=k\langle\, a, b, c, d \mid  a^n=0=d^n, b^n=1=c^n, ba=qab, dc=qcd, &\\ db=qbd, bc=cb, ca=qac, da-qad=1-bc \, \rangle &
\end{align*}
generated by $a$ and $b$, so that $U_n(q)=k\langle a, b \mid a^n=0, b^n=1, ba=qab \rangle$. Then for the Hopf subalgebra pair $U_n(q)\subseteq D(U_n(q))$ the corresponding quotient module is 
\begin{equation*}
Q=D(R)/R^+ D(R),
\end{equation*}
with $R = U_n(q)$ and $R^+=\ker\,\epsilon=(1-b)R+aR$. 

The right $R$-module $Q$ is $n^2$-dimensional and has basis $\{ \ee{i, j} \mid  0\leq i, j\leq n-1\}$, where $\ee{i, j}=c^i d^j+R^+ D(R)\in Q$ and $\ee{i+n, j}=\ee{i, j}$ for all $i\in\ZZ$, as $c^n=1$. The right action of $R$ on $Q$ is determined by:
\begin{align}\label{E:b_action}
 \ee{i, j}.b &=q^j\ee{i, j},\\ \label{E:a_action}
 \ee{i, j}.a &=(j)_q(\ee{i, j-1}-q^{j-1}\ee{i+1, j-1}),
\end{align}
for all $i\in\ZZ$ and all $0\leq j\leq n-1$, with the additional convention that $\ee{i, -1}=0$.

The following computation will be used frequently so we record it in a lemma. The proof is routine and is omitted. 

\begin{lemma}\label{L:a_action}
Fix $i_0\in\ZZ$ and let $X=\sum_{i=i_0}^{i_0 +n-1} \lambda_i \ee{i, j}$ be an element of $Q$, with $\lambda_i\in k$ for all $i_0\leq i\leq i_0 +n-1$. Define $\lambda_{i_0 -1}=\lambda_{i_0 +n-1}$. Then
\begin{equation*}
X.a=(j)_q\sum_{i=i_0}^{i_0 +n-1} (\lambda_i -q^{j-1}\lambda_{i-1})\ee{i, j-1}.
\end{equation*}
\end{lemma}

In order to describe the  Krull-Schmidt decomposition of the $R$-module $Q$, we define a new basis for $Q$. Let $\ell=1, \ldots, n$ and define
\begin{align*}
\uu{\ell}&=\sum_{i=1}^n q^{i(\ell-1)}{i+n-2\choose n-1}_q \ee{i, \ell-1},\\[5pt]
\ww{\ell}&=\sum_{i=1}^n q^{-i(\ell+1)}{i+\ell-2\choose \ell-1}_q \ee{i, n-1}.
\end{align*}
It follows from Lemma~\ref{L:qbinom} that $\uu{\ell}=q^{\ell-1}\ee{1, \ell-1}$, but it will be more convenient to use the defining expression for $\uu{\ell}$ given above. Also note that $\uu{n}=\ww{n}$.

\begin{lemma}\label{L:a_action_t_w}
For all $1\leq\ell\leq n$ and all $0\leq r\leq \ell-1$ the action of $a\in U_n(q)$ on the elements $\uu{\ell}$ and $\ww{\ell}$ is given by the following expressions:
\begin{enumerate}
\item[\textup{(a)}] $\displaystyle \uu{\ell}.a^r=\frac{(\ell-1)!_q}{(\ell-r-1)!_q}\sum_{i=1}^n q^{i(\ell-1)}{i+n-2-r\choose n-1-r}_q \ee{i, \ell-1-r}$,\\[5pt]
\item[\textup{(b)}] $\displaystyle \ww{\ell}.a^r=\frac{(n-1)!_q}{(n-r-1)!_q}\sum_{i=1}^n q^{-i(\ell+1)}{i+\ell-2-r\choose \ell-1-r}_q \ee{i, n-1-r}$.
\end{enumerate}
Moreover, the elements $\uu{\ell}.a^{\ell-1}$ with $1\leq \ell\leq n$ are a basis of $\bigoplus_{i=0}^{n-1} k \ee{i, 0}$,
$\uu{\ell}.a^r\neq0$,  $\ww{\ell}.a^r\neq 0$ for all $0\leq r\leq \ell-1$ and $\uu{\ell}.a^\ell=0=\ww{\ell}.a^\ell$.
\end{lemma}
\begin{proof}
The proof is by induction on $r$, with the case $r=0$ being clear. So assume the result for $0\leq r\leq\ell-2$ with $\ell\geq 2$. Then, by Lemma~\ref{L:a_action},
\begin{align*}
 \uu{\ell}.a^{r+1} &=\frac{(\ell-1)!_q}{(\ell-r-1)!_q}\sum_{i=1}^n q^{i(\ell-1)}{i+n-2-r\choose n-1-r}_q \ee{i, \ell-1-r}.a\\
 &=\frac{(\ell-1)!_q}{(\ell-r-2)!_q}\sum_{i=1}^n q^{i(\ell-1)}\left({i+n-2-r\choose n-1-r}_q -q^{-r-1}{i+n-3-r\choose n-1-r}_q\right)\ee{i, \ell-2-r},
\end{align*}
adopting the convention that ${n-2-r\choose n-1-r}_q={2n-2-r\choose n-1-r}_q$, as required by Lemma~\ref{L:a_action}. But then by Lemma~\ref{L:qbinom} we conclude that ${n-2-r\choose n-1-r}_q={2n-2-r\choose n-1-r}_q=0$, which is consistent with our convention that ${k\choose k+1}_q=0$ for $k\geq 0$. So we can apply \eqref{E:recurrence} to obtain
\begin{align*}
 {i+n-2-r\choose n-1-r}_q-q^{n-r-1} {i+n-r-3\choose n-1-r}_q={i+n-r-3\choose n-r-2}_q \quad \mbox{for all $1\leq i\leq n$,}
\end{align*}
and thus, using $q^n=1$, we establish the $r+1$ case of (a):
\begin{align*}
 \uu{\ell}.a^{r+1} &=\frac{(\ell-1)!_q}{(\ell-r-2)!_q}\sum_{i=1}^n q^{i(\ell-1)} {i+n-r-3\choose n-r-2}_q \ee{i, \ell-2-r}.
\end{align*}

For (b) we have, as above,
\begin{align*}
 \ww{\ell}.a^{r+1} &=\frac{(n-1)!_q}{(n-r-1)!_q}\sum_{i=1}^n q^{-i(\ell+1)}{i+\ell-2-r\choose \ell-1-r}_q \ee{i, n-1-r}.a\\
 &=\frac{(n-1)!_q}{(n-r-2)!_q}\sum_{i=1}^n q^{-i(\ell+1)}\left({i+\ell-2-r\choose \ell-1-r}_q-q^{\ell-1-r}{i+\ell-3-r\choose \ell-1-r}_q\right)\ee{i, n-2-r},
\end{align*}
with the convention that ${\ell-2-r\choose \ell-1-r}_q={n+\ell-2-r\choose \ell-1-r}_q$, as indicated in Lemma~\ref{L:a_action}. Again by Lemma~\ref{L:qbinom} we conclude that ${\ell-2-r\choose \ell-1-r}_q={n+\ell-2-r\choose \ell-1-r}_q=0$, which is consistent with our convention that ${k\choose k+1}_q=0$ for $k\geq 0$. So we can apply \eqref{E:recurrence} to obtain
\begin{align*}
{i+\ell-2-r\choose \ell-1-r}_q-q^{\ell-1-r} {i+\ell-3-r\choose \ell-1-r}_q={i+\ell-3-r\choose \ell-2-r}_q \quad \mbox{for all $1\leq i\leq n$,}
\end{align*}
and conclude the inductive step:
\begin{align*}
 \ww{\ell}.a^{r+1} &=\frac{(n-1)!_q}{(n-r-2)!_q}\sum_{i=1}^n q^{-i(\ell+1)}{i+\ell-3-r\choose \ell-2-r}_q \ee{i, n-2-r}.
\end{align*}
Note that 
\begin{align*}
 \uu{\ell}.a^{\ell-1} &=(\ell-1)!_q\sum_{i=1}^n q^{i(\ell-1)}{i+n-1-\ell \choose n-\ell}_q \ee{i,0}\\
 &=(\ell-1)!_q\sum_{i=1}^\ell q^{i(\ell-1)}{i+n-1-\ell \choose n-\ell}_q \ee{i,0} \quad \mbox{by Lemma~\ref{L:qbinom},}\\
 &=(\ell-1)!_q\left( q^{\ell(\ell-1)}{n-1 \choose n-\ell}_q \ee{\ell,0}+\mathsf{x}_\ell\right),
\end{align*}
with $\mathsf{x}_\ell$ in the $k$-span of $\{ \ee{i, 0} \mid 1\leq i\leq \ell-1\}$. In particular, as $(\ell-1)!_q q^{\ell(\ell-1)}{n-1 \choose n-\ell}_q\neq 0$ for all $1\leq \ell\leq n$ it is evident that the $n$ elements $\uu{\ell}.a^{\ell-1}$ with $1\leq \ell\leq n$ are linearly independent, and thus form a basis of $\bigoplus_{i=1}^n k \ee{i, 0}=\bigoplus_{i=0}^{n-1} k \ee{i, 0}$. Moreover, as $\uu{\ell}.a^{\ell-1}$ is in the $k$-span of $\{ \ee{i, 0} \mid 1\leq i\leq \ell\}$ and $\ee{i, 0}.a=0$ for all $i\in\ZZ$, we also have $\uu{\ell}.a^\ell=0$.

Now using (b) we obtain $\ww{\ell}.a^{\ell-1}=\frac{(n-1)!_q}{(n-\ell)!_q}\sum_{i=1}^n q^{-i(\ell+1)} \ee{i, n-\ell}\neq 0$ and another application of Lemma~\ref{L:a_action} yields $\ww{\ell}.a^{\ell}=0$.
\end{proof}

We are ready to introduce a new basis for $Q$ that leads to its decomposition as an $U_n(q)$-module.

\begin{prop}\label{P:basis_Q}
The elements
\begin{equation}\label{E:basis_t}
\uu{\ell}.a^r, \quad \mbox{with $\ell=1, \ldots, n$ and $r=0, \ldots, \ell-1$},
\end{equation}
and the elements 
\begin{equation}\label{E:basis_w}
\ww{\ell'}.a^{r'}, \quad \mbox{with $\ell'=1, \ldots, n-1$ and $r'=0, \ldots, \ell'-1$},
\end{equation}
together form a basis of $Q$.
\end{prop}

\begin{proof}
Since $\dim_k Q=n^2$ and there are $n^2$ elements listed in \eqref{E:basis_t} and \eqref{E:basis_w}, it suffices to show that the elements in \eqref{E:basis_t} and \eqref{E:basis_w}, when taken together, are linearly independent. 

Notice first that, by \eqref{E:b_action}, 
\begin{equation*}
(\uu{\ell} .a^r).b=q^{\ell-r-1}(\uu{\ell} .a^r) \quad \mbox{and}\quad (\ww{\ell'} .a^{r'}).b=q^{-r'-1}(\ww{\ell'} .a^{r'}),
\end{equation*}
and for $\ell, \ell', r, r'$ in the ranges specified in \eqref{E:basis_t} and \eqref{E:basis_w}, the elements $\uu{\ell} .a^r$ and $\ww{\ell'} .a^{r'}$ are nonzero, as observed in Lemma~\ref{L:a_action_t_w}, and furthermore $0\leq \ell-r-1\leq n-1$ and $-(n-1)\leq -r'-1\leq -1$. Thus these elements are eigenvectors for the action of $b$ and it suffices to show that elements from \eqref{E:basis_t} and \eqref{E:basis_w} associated to the same eigenvalue of $b$ are linearly independent.

So fix $0\leq j\leq n-1$. The eigenvectors for $b$ among \eqref{E:basis_t} and \eqref{E:basis_w} corresponding to the eigenvalue $q^j$ are the $n$ elements of the form $\uu{\alpha}.a^{\alpha-j-1}$ and $\ww{\beta}.a^{n-j-1}$, for $j+1\leq \alpha\leq n$ and $n-j\leq \beta\leq n-1$ (in case $j=0$ there are no elements of the form $\ww{\beta}.a^{k}$). Suppose, by way of contradiction, that there exists $0\leq j\leq n-1$ and scalars $\lambda_1, \ldots, \lambda_n\in k$, not all $0$, such that 
\begin{equation}\label{E:dependence_rel}
\sum_{i=1}^j \lambda_i \ww{n-i}.a^{n-j-1} + \sum_{i=j+1}^n \lambda_i \uu{i}.a^{i-j-1}=0.
\end{equation}
We can choose $j$ minimal with this property. Then $j\geq 1$, because for $j=0$ we get the elements $\uu{\ell}.a^{\ell-1}$ with $1\leq \ell\leq n$, which by Lemma~\ref{L:a_action_t_w} are linearly independent. Applying $a$ to  \eqref{E:dependence_rel} yields
\begin{align*}
 0&=\sum_{i=1}^j \lambda_i \ww{n-i}.a^{n-j} + \sum_{i=j+1}^n \lambda_i \uu{i}.a^{i-j}=\sum_{i=1}^{j-1} \lambda_i \ww{n-i}.a^{n-(j-1)-1} + \sum_{i=j+1}^n \lambda_i \uu{i}.a^{i-(j-1)-1},
\end{align*}
since $\ww{n-j}.a^{n-j}=0$. The minimality assumption on $j$ implies that $\lambda_i=0$ for all $i\neq j$. Hence, from \eqref{E:dependence_rel} we get $\lambda_j \ww{n-j}.a^{n-j-1}=0$. Since $\ww{n-j}.a^{n-j-1}\neq0$, we deduce that also $\lambda_j=0$, contradicting the assumption that not all of the $\lambda_i$ are $0$. This established our claim.
\end{proof}

In our next result, we obtain indecomposable $U_n(q)$-submodules of $Q$ from the new basis elements introduced in Proposition~\ref{P:basis_Q}.

\begin{prop}\label{P:indec_from_basis}\hfill
\begin{enumerate}
\item[\textup{(a)}] For $\ell=1, \ldots, n$, $\displaystyle\bigoplus_{r=0}^{\ell-1} k\, \uu{\ell}.a^r$ is an $\ell$-dimensional right $U_n(q)$-submodule of $Q$ isomorphic to the indecomposable module $P_2 J^{n-\ell}$.
\item[\textup{(b)}] For $\ell'=1, \ldots, n-1$, $\displaystyle\bigoplus_{r'=0}^{\ell'-1} k\, \ww{\ell'}.a^{r'}$ is an $\ell'$-dimensional right $U_n(q)$-submodule of $Q$ isomorphic to the indecomposable module $P_{2+\ell'} J^{n-\ell'}$, where by convention $P_{n+1}=P_1$.
\end{enumerate}
\end{prop}

\begin{proof}
For the proof of (a), write  $\displaystyle V_\ell=\bigoplus_{r=0}^{\ell-1} k\, \uu{\ell}.a^r$. By Proposition~\ref{P:basis_Q}, it is clear that this sum is indeed direct and thus $\dim_k V_\ell=\ell$. Moreover, $V_\ell$ is a right $U_n(q)$-submodule of $Q$ because $\uu{\ell}.a^r$ is an eigenvector for $b$ and $\left(\uu{\ell}.a^{\ell-1}\right).a=\uu{\ell}.a^{\ell}=0$. 

Under the action of $b$ on $V_\ell$, the eigenvectors $\uu{\ell}.a^r$ have corresponding eigenvalue $q^{\ell-r-1}$, for $r=0, \ldots,\ell-1$. Thus the assignment $\uu{\ell}.a^r\mapsto e_2.a^{n-\ell+r}$, for $r=0, \ldots,\ell-1$, gives the required isomorphism $V_\ell\longrightarrow P_2 J^{n-\ell}$.

Part (b) is entirely analogous and we just describe the isomorphism $W_{\ell'}\longrightarrow P_{2+\ell'} J^{n-\ell'}$, where $\displaystyle W_{\ell'}=\bigoplus_{r'=0}^{\ell'-1} k\, \ww{\ell'}.a^{r'}$, which is $\ww{\ell'}.a^{r'}\mapsto e_{2+\ell'}.a^{n-\ell'+r'}$, for $r'=0, \ldots,\ell'-1$.
\end{proof}

The Krull-Schmidt decomposition of $Q$ now falls out easily.

\begin{theorem}\label{T:KSdecQ}
We have the following decomposition of $Q$ into indecomposables, as a right $U_n(q)$-module:
\begin{align*}
 Q&\cong\bigoplus_{\ell=1}^{n} P_2 J^{n-\ell} \oplus \bigoplus_{\ell'=1}^{n-1} P_{2+\ell'} J^{n-\ell'}\\
 &\cong P_2 \oplus P_2 J \oplus  \cdots \oplus P_2 J^{n-1}\oplus P_3 J^{n-1}\oplus P_4 J^{n-2}\oplus\cdots\oplus P_n J^2\oplus P_1 J.
\end{align*}
\end{theorem}
\begin{proof}
 The decomposition follows immediately from Proposition~\ref{P:basis_Q} and Proposition~\ref{P:indec_from_basis}.
\end{proof}

For the remainder of this section and the next, we will adopt the notation $M(\ell, i)$ used in~\cite{COZ} for the indecomposable $U_n(q)$-modules. This will make it more convenient for the reader to check any additional details in~\cite{COZ} concerning the Green ring $r(U_n(q))$ of the Taft algebra $U_n(q)$ and the tensor product of indecomposable $U_n(q)$-modules. The $P_iJ^r$ notation may be recovered  through the isomorphism 
\begin{equation*}
M(\ell, i)\cong P_{\ell-i+1}J^{n-\ell}, \quad \mbox{with indices mod $n$.}
\end{equation*}
 Note that  the authors in \cite{COZ} consider left modules in contrast to our right modules; one goes from one to the other by changing from $R$ to $R^{\textup{op}}$, and  from $q$ to $q^{-1}$, so that the results are unchanged. 
\begin{cor}
In the notation of \cite{COZ}, the isoclass of $Q$ satisfies the equation
\begin{equation}
[Q] = 1 + a + (a^{n-1} + a)x + (a^{n-2}+a)u_3 + \cdots + (a^2 +a)u_{n-1} + au_n
\end{equation}
in the Green ring $r(U_n(q))$.
\end{cor}
\begin{proof}
Recall from \cite{COZ} that $r(U_n(q))$ is generated by the simple   $a = [M(1,n-1)]  = [P_3J^{n-1}]$, and
the $2$-dimensional indecomposable,  $x = [M(2,0)] = [P_3J^{n-2}]$.  The Fibonacci-like polynomials $u_{\ell}(a,x)$  satisfy 
$u_{\ell + 1} = xu_{\ell} - au_{\ell - 1}$, where $u_1 = 1$ and $u_2 = x$, and satisfy $u_{\ell} = [M(\ell,0)]$. Thus, $[M(\ell,r)] = a^{n-r}u_{\ell}$. The rest follows from the Theorem, reinterpreted as Eq.~(\ref{eq: coznotation}) below.
\end{proof}
\section{Decomposition of $Q\otimes Q$}

The aim of this section is to show that all indecomposable $U_n(q)$-modules occur in the Krull-Schmidt decomposition of $Q\otimes Q$. We use the notation in \cite{COZ} as mentioned above. 
For the reader's convenience we record below the results from \cite[Section 3]{COZ} which are most relevant to us. Recall that $M(\ell, r)=M(\ell, r+n)$ for all $1\leq \ell\leq n$ and $r\in\mathbb{Z}$.

\begin{theorem}[{\cite[Section 3]{COZ}}]\label{T:COZ14}
\label{thm-COZ} Let $1\leq \ell, \ell'\leq n$, $r, r'\in\mathbb{Z}$, $\ell_0 = \min\{\ell, \ell'\}$ and $\ell_1 = \max\{\ell, \ell'\}$. The following hold:
\begin{itemize}
\item[\textup{(a)}] $\displaystyle M(\ell, r)\otimes M(\ell', r')\cong M(\ell', r')\otimes M(\ell, r)$.
\item[\textup{(b)}] $\displaystyle M(\ell, r)\otimes M(1, r')\cong M(\ell, r+r')$.
\item[\textup{(c)}] If $\ell \geq 2$, then $\displaystyle M(\ell, r)\otimes M(n, r')\cong \bigoplus_{i=1}^\ell M(n,r+r'+i-\ell)$.
\item[\textup{(d)}] If $\ell +\ell' \leq n$, then $\displaystyle M(\ell, r)\otimes M(\ell', r')\cong \bigoplus_{i=1}^{\ell_0} M(\ell_1-\ell_0-1+2i,r+r'+i-\ell_0)$.
\item[\textup{(e)}] If $\ell, \ell'< n$ and $\ell +\ell' \geq n$, then 
\begin{align*}
 M(\ell, r)\otimes M(\ell', r')\cong  \left(\bigoplus_{i=1}^{n-\ell_1} M(\ell_1-\ell_0-1+2i,r+r'+i-\ell_0)\right) \ &\\
   \qquad \oplus \left(\bigoplus_{j=1}^{\ell+\ell'-n} M(n,r+r'+1-j)\right).&
\end{align*}
\end{itemize}
\end{theorem}

Given the decomposition 
\begin{equation}
\label{eq: coznotation}
Q\cong\bigoplus_{\ell=1}^{n} M(\ell, \ell-1) \oplus \bigoplus_{\ell'=1}^{n-1} M(\ell', n-1)
\end{equation}
from Theorem~\ref{T:KSdecQ} and the distributivity of $\oplus$ relative to $\otimes$, it is enough to show that for any integers $1\leq x\leq n$ and $0\leq y\leq n-1$, the module $M(x, y)$ is isomorphic to a direct summand of $M(i, j)\otimes M(i', j')$ for some $(i, j), (i', j')\in \mathcal{J}$, where 
\begin{equation*}
\mathcal{J}=\{ (\ell, \ell-1) \mid 1\leq \ell\leq n\}\cup \{ (\ell, n-1) \mid 1\leq \ell\leq n-1\}.
\end{equation*}

We will consider four cases.

\medskip 

\textbf{Case 1:} Assume that $y=n-1$ or $x=y+1$. Then $(x, y)\in\mathcal{J}$, $(1, 0)\in\mathcal{J}$ and $M(x, y)\otimes M(1, 0)\cong M(x, y)$, by Theorem \ref{T:COZ14}(b).

\medskip 

\textbf{Case 2:} Assume that $x=n$. Then as $(n, n-1)\in\mathcal{J}$ and $M(n, n-1)\otimes M(n, n-1)\cong\bigoplus_{i=1}^n M(n, i-2)$, by Theorem \ref{T:COZ14}(c), we deduce that $M(n, y)$ is a direct summand of $M(n, n-1)\otimes M(n, n-1)$ for any $y$, since as $i$ varies from $1$ to $n$, $i-2$ runs through all residue classes modulo $n$.

\medskip 

\textbf{Case 3:} Assume that $x\leq y\leq n-2$, with $n\geq 3$. Define $k=n+x-y-2$ and $\ell=n-y-1$. Then $1\leq x\leq k\leq n-2$ and $1\leq \ell\leq k$, so that $(k, n-1), (\ell, n-1)\in \mathcal{J}$. We will show, using Theorem \ref{T:COZ14}, that $M(x, y)$ is a direct summand of $M(k, n-1)\otimes M(\ell, n-1)$. 

Indeed, we have $\ell_0 = \min\{k, \ell\}=\ell$ and $\ell_1 = \max\{k, \ell\}=k$; in particular, $\ell_0\geq 1$ and $n-\ell_1\geq 1$. If $k+\ell\leq n$ (respectively, $k+\ell> n$),  Theorem \ref{T:COZ14}(d) (respectively, Theorem \ref{T:COZ14}(e)) with $i=1$ shows that $M(k-\ell+1, 2n-1-\ell)$ is a direct summand of $M(k, n-1)\otimes M(\ell, n-1)$. The claim follows since 
$k-\ell+1=x$ and $2n-1-\ell=y+n$, which is congruent to $y$ modulo $n$.

\medskip 

\textbf{Case 4:} Assume that $y+2\leq x\leq n-1$, with $n\geq 3$. Define $k=n-y-2$ and $\ell=n+y-x+1$. Then $1\leq n-x\leq k\leq n-2$ and $2\leq n-k\leq \ell\leq n-1$, so that $(k, n-1), (\ell, \ell-1)\in \mathcal{J}$. We will show that $M(x, y)$ is a direct summand of $M(k, n-1)\otimes M(\ell, \ell-1)$. 

Since $k+\ell\geq n$, Theorem \ref{T:COZ14}(e) applies and by taking $i=n-\ell_1\geq 1$ in the formula, we conclude that 
$M(-(\ell_0+\ell_1)+2n-1,2n-2+\ell-(\ell_0+\ell_1))$ is a direct summand of $M(k, n-1)\otimes M(\ell, \ell-1)$, where $\ell_0 = \min\{k, \ell\}$ and $\ell_1 = \max\{k, \ell\}$. Using $\ell_0+\ell_1=k+\ell$ we see that $-(\ell_0+\ell_1)+2n-1=-(2n-1-x)+2n-1=x$ and $2n-2+\ell-(\ell_0+\ell_1)=2n-2-k=y+n$, which is congruent to $y$ modulo $n$.

We have thus proved the following.

\begin{theorem}\label{T:decQQ}
All indecomposable $U_n(q)$-modules occur in the Krull-Schmidt decomposition of $Q\otimes Q$ as a right $U_n(q)$-module.
\end{theorem}


We remark that the endomorphism ring $\End Q^{\otimes 2}$ of the $U_n(q)$-module $Q \otimes Q$ is therefore 
Morita equivalent to the Auslander algebra of $U_n(q)$ \cite{SY}. 

\begin{cor}
For the Taft algebra $U_n(q)$ in its Drinfeld double, the depth of $Q$ is $d(Q_{U_n(q)}) = 2$ and the minimum even depth $$d_{ev}(U_n(q), D(U_n(q)) = 6.$$
\end{cor}

\medskip

\subsection{Example: the $n = 2$ Sweedler Hopf algebra case}\label{E:51}
We recapitulate the computations  in the toy model case when $n = 2$, $k=\mathbb{C}$ and $q = -1$, the Sweedler Hopf algebra, with some additional remarks.  The Hopf algebra structure of $R=H_4$  is then given by $H_4 = \mathbb{C} \bra a,b,\| a^2 = 0, b^2 = 1, ab = -ba \ket$
with coproduct \begin{equation}
\label{eq: cop}
\cop(b) = b \otimes b, \ \ \ \ \cop(a) = a \otimes b + 1 \otimes a,
\end{equation}
and counit $\eps(b) = 1$, $\eps(a) = 0$, and antipode $S(b) = b$, $S(a) = -ab$. The radical ideal
is $J = aH_4$.  
In all, there are four indecomposables $P_1 = e_1H_4$, $P_2 = e_2 H_4$, where $e_1 = (1 + b)/2$, $e_2 = (1-b)/2$,
and $S_1 = P_1/P_1 a$, $S_2 = P_2 / P_2 a$, up to isomorphism. 

The quantum double algebra is given in terms of generators and relations by 
\begin{eqnarray*}
\label{eq: double}
D(H_4) & = & \mathbb{C} \bra a,b,c,d \| a^2 = 0 = d^2, b^2 = 1 = c^2, bc = cb, ab = -ba, \\
& & \ \ \ \ \ \ \ \ \ \ \ \ cd = -dc, ac = -ca, bd = -db, da + ad = 1 - bc \ket
\end{eqnarray*}
which is $16$-dimensional with basis $\{ a^ib^jc^rd^s \| i,j,r,s = 0,1 \}$. Note the obvious symmetry given by the involutory algebra automorphism exchanging $a$ with $d$ and $b$ with $c$.
The $\mathbb{C}$-algebra $A = D(H_4)$ has a basic set of two \textit{central} orthogonal idempotents $e_1 = \frac{1 - bc}{2}$ and
$e_2 = \frac{1 + bc}{2}$.  Note that $ad + da = 2e_1$, $e_2 ad = -e_2da$ and that
$e_{11} := e_1ad/2$ is a noncentral idempotent.  It follows that $A = e_1A \oplus e_2 A$,
where  our computations give $e_1 A \cong M_2(\mathbb{C}) \otimes \mathbb{C} \Z_2$, where the other matrix units are given by
$e_{22} := e_1da/2$, $e_{12} := e_1 a/ \sqrt{2}$ and $e_{21} := e_1d/\sqrt{2}$.  The block
$e_1A$ is in fact a bi-ideal in the Hopf algebra structure on $A$.  The coproduct formulas and the representation theory of $D(H_4)$ may be found in \cite{HXC}, where it may also be noted that the symmetry is a Hopf algebra isomorphism of copies of $H_4$ within $D(H_4)$.  

The block $e_2A$ is isomorphic as Hopf algebras to $H_8 := \overline{U_i}(sl_2(\mathbb{C}))$, the small quantum group at the fourth root of unity $i$ generated by a grouplike $K$, and two commuting nilpotent elements $E,F$
that anti-commute with $K$ (all of order $2$, see for example \cite[Example 4.9]{K2014}).  The isomorphism is given by
$e_2 \mapsto 1$, $e_2 b \mapsto K$, $e_2 a \mapsto E$ and $e_2 bd \mapsto F$, the rest being an exercise; alternatively, it is part of a more general fact for the Drinfeld double of any half quantum group and any small quantum group sharing the same (odd or half-even) root-of-unity \cite[Chapter 9]{CK}.  

The quotient module $Q = D(H_4)/ H_4^+ D(H_4)$ is computed by noting that $H_4^+ = aH_4 + (1 - b)H_4$, thus has basis $\{ \overline{1}, \overline{c}, \overline{d}, \overline{cd} \}$.  The right $H_4$-module structure is given by
\begin{eqnarray*}
\overline{1}. a & = & \overline{0} = \overline{c}. a \\
\overline{1} . b & = & \overline{1}, \ \ \overline{c}.b = \overline{c} \\
\overline{d}.a & = & \overline{1}-\overline{c}, \ \ \overline{d}.b = -\overline{d} \\
\overline{cd}.a & = & \overline{c}-\overline{1}, \ \ \overline{cd}.b = -\overline{cd}
\end{eqnarray*}
which are short computations with the relations of $D(H_4)$, or follow from Eqs.~(\ref{E:b_action}) and~(\ref{E:a_action}).

Note the following submodules of $Q_{H_4}$.  The $2$-dimensional submodule with basis $\{ \overline{d}, \overline{1}-\overline{c} \}$ isomorphic to $P_2$.  The $1$-dimensional submodule spanned by $\overline{c}$, isomorphic to $S_1$.  The $1$-dimensional submodule spanned by $\overline{d} + \overline{cd}$, isomorphic to $S_2$. Then it is easy to see that $Q_{H_4} \cong S_1 \oplus S_2 \oplus P_2$.  Note that $P_1$ is not
a  constituent of $Q$.  

Since $S_1$ is the unit module, the tensor-square $Q \otimes Q$ has summand $P_2 \otimes P_2$, which may be computed directly or using Theorem~\ref{thm-COZ}, as 
\begin{equation}
P_2 \otimes P_2 \cong P_1 \oplus P_2
\end{equation}
We see that $Q \otimes Q$, unlike $Q$, has all four indecomposables of $H_4$ as nonzero constituents, so its depth as an $H_4$-module is $2$.

\section*{Acknowledgements}

Samuel A.\ Lopes was partially supported by CMUP (UID/MAT/00144/2013), which is funded by FCT (Portugal) with national (MEC) and European structural funds (FEDER), under the partnership agreement PT2020.

\end{document}